\documentclass[a4paper,10pt]{amsart}
\usepackage{amsthm,amsmath,amsfonts,amssymb}
\usepackage{bm}
\usepackage{mathtools,mathrsfs}
\usepackage{todonotes}
\usepackage{pgfplots}
\usepackage{mathrsfs}
\usepackage[pdfdisplaydoctitle,colorlinks,breaklinks,urlcolor=blue,linkcolor=blue,citecolor=blue]{hyperref} 
\usepackage[nameinlink,capitalise]{cleveref}

\newcommand{\N}{\mathbb{N}}

\newcommand{\R}{\mathbb{R}}

\DeclareMathOperator{\var}{Var}
\DeclareMathOperator{\sinc}{sinc}

\renewcommand{\epsilon}{\varepsilon}
\renewcommand{\setminus}{\smallsetminus}

\newcommand{\one}{\bm{1}}

\newcommand{\set}[1]{\left\{#1\right\}}
\newcommand{\pa}[1]{\left(#1\right)}
\newcommand{\bra}[1]{\left[#1\right]}
\newcommand{\abs}[1]{\left|#1\right|}
\newcommand{\norm}[1]{\left\|#1\right\|}
\newtheorem{theorem}{Theorem}

\newtheorem{lemma}[theorem]{Lemma}

\theoremstyle{remark}

\numberwithin{equation}{section}

\newenvironment{acknowledgements}{%
  \begin{abstract}
}{%
  \end{abstract}
}

\title[Missing Central Limit Theorems]{The Missing Central Limit Theorems for Local Functionals of Berry's Random Wave Model}
\author[F. Grotto]{Francesco Grotto}
\address{Università di Pisa, Dipartimento di Matematica, 5 Largo Bruno Pontecorvo, 56127 Pisa, Italia.}
\email{francesco.grotto at unipi.it}
\date\today

\begin{document}

\begin{abstract}
    Central Limit Theorems for integrals of third degree Hermite polynomials of Berry's random wave model on increasingly large domains are proved in dimensions 2 and 3. These were the missing cases for a complete description of limit theorems for integral functionals of monochromatic random waves based on the Wiener chaos decomposition.
\end{abstract}

\maketitle

\section{Introduction}\label{sec:intro}

Berry's model is the centered Gaussian random field $u_\lambda(x)$, $\lambda>0$, $x\in \R^d$, with covariance function
\begin{equation}\label{eq:covberry}
    E\bra{u_\lambda(x)u_\lambda(y)}=
    j_d(\lambda|x-y|),\quad
    x,y\in \R^d,
\end{equation}
where $j_d$ is the Fourier transform of the surface measure $\sigma_{d-1}$ of the unit sphere $S^{d-1}=\set{|x|=1}\subset\R^d$,
\begin{equation}\label{eq:jd}
     j_d(|x|) =  \frac1{\omega_{d-1}}\int_{S^{d-1}} e^{i x\cdot \theta} d\sigma_{d-1}(\theta)=
     \frac{\Gamma(\nu+1)2^\nu}{|x|^\nu}J_\nu(|x|), \quad \nu=\frac{d}{2}-1,
\end{equation}
$J_\nu$ being the Bessel function of the first kind and order $\nu$. 
This random field was first introduced by its eponym \cite{Berry1977,Berry2002,Berry2002a} in the context of Quantum Chaos, as a model for describing high-frequency behavior of wavefunctions in quantum billiards whose classical analogue is chaotic. Berry's model is the monochromatic Fourier projection of an isotropic Gaussian random field, hence its mathematical relevance and that of its geometric functionals.

Consider local functionals of Berry's model of the form
\begin{equation}\label{eq:functional}
    F(u_\lambda )=\int_D f(u_\lambda (x))dx,
\end{equation}
where $f$ is a function of the field $u_\lambda(x)$ evaluated at points of a domain $D\subseteq \R^d$. The most relevant questions concern the asymptotic behavior as $\lambda\to \infty$ or $|D|\to \infty$, which are equivalent if for instance $D=B_R$, $R\to \infty$ is a Euclidean ball of increasingly larger radius (by scaling).
Such limit theorems, together with those on related models on spheres, tori and other underlying geometric structures, are the object of extensive literature, the review of which falls outside the scope of this article: the recent works \cite{maininourdin,maini25} provide a proper overview and complete references. In what follows, I shall only mention directly related results.

Wiener chaos expansion provides a hands-on approach to the study of functionals like \eqref{eq:functional}. Since $u_\lambda(x)$ is a Gaussian random variable the decomposition of $f$ into the orthogonal basis of Hermite polynomials,
\begin{equation*}
    H_q(t)=(-1)^q \frac{\phi^{(q)}(t)}{\phi(t)},\quad \int_\R H_q(t)H_{q'}(t)\phi(t)dt=q!\delta_{q=q'},\quad
    \phi(t)=\frac{e^{-t^2/2}}{\sqrt{2\pi}},
\end{equation*}
induces that of $F(u_\lambda)$, and one is reduced to study
\begin{equation*}
    h(q,\lambda,D)=\int_D H_q(u_\lambda (x))dx, \quad q\geq 1.
\end{equation*}
Fluctuations of $F(u_\lambda)$ are therefore reduced to those of $h(q,\lambda,D)$ and appropriate tail bounds on chaos projections. This allows direct applications of the Malliavin-Stein method. Generalized functions $f$ (for instance $f=\delta_0$ computes the occupation or Leray measure of zero level sets) can also be treated if their coefficients in the chaos expansion are suitably summable.

I shall focus on the isotropic limit $R\to \infty$ for $D=B_R$. Second moments of
\begin{equation*}
    h(q,R)=h(q,\lambda=1,D=B_R)
    =\int_{B_R} H_q(u_1(x))dx
\end{equation*}
(having the same law of $R^d h(q,\lambda=R,D=B_1)$) satisfy
\begin{equation*}
    \var h(q,R)=
    \begin{cases}
        C_{2,d}R^{d+1}(1+o(1)), & q=2,\\[3pt]
        C_{4,2}R^2\log R\,(1+o(1)), & d=2,\ q=4,\\[3pt]
        C_{q,d}R^d(1+o(1)), & q\geq 3,\ (d,q)\neq(2,4),
    \end{cases}
    \qquad
    \text{as }R\to\infty,
\end{equation*}
for finite positive constants $C_{q,d}$, given in the last case by the improper --but finite and positive-- integrals
\begin{equation}\label{eq:constants}
    C_{q,d}=\frac{q!}{d}\omega_{d-1}^2\int_0^\infty j_d(t)^q t^{d-1}dt
\end{equation}
(see \cite{Nourdin2019}, leading constants were discussed in \cite{Grotto2023b}). Then, the Fourth Moment Theorem \cite{Nourdin2012} allows to establish Central Limit Theorems for $h(q,\lambda,D)$ provided that their fourth cumulants are negligible compared to squared variances. 
An equivalent condition, after writing $u_\lambda(x)=\int_{S^{d-1}} e^{i\lambda \xi\cdot x} \eta(d\xi)$,
\begin{equation*}
    h(q,R)=\int_{(S^{d-1})^q} Q_R(x_1+\cdots+x_q) \eta^q(dx_1,\dots,dx_q),\quad 
    Q_R(x)=\int_{B_R} e^{i \xi x}d\xi,
\end{equation*}
as $q$-fold Wiener-It\^o integrals with respect to Hermitian complex white noise $\eta$ on $S^{d-1}$ with reference measure $\frac{\sigma_{d-1}}{\omega_{d-1}}$ (in the sense of \cite[Theorem 7.52]{Janson1997}), is that kernels satisfy the contraction estimates
\begin{multline}\label{eq:contractions}
    \norm{Q_R(x_1+\cdots+x_q)\otimes_{r}Q_R(x_1+\cdots+x_q)}_2^2\\
    =o\pa{\var h(q,R)^2},\quad R\to\infty,\,r=1,\dots,q-1,
\end{multline}
where $\otimes_r$ denotes integration over $r$ of the $q$ (symmetric) variables $x_i$ and tensor product over the others. This asymptotic relation was proved in \cite{Nourdin2019,maininourdin} for all orders $q$ and dimensions $d$ except for the case $q=3$ in $d=2,3$, in which the latter estimate is not immediately reduced to tail bounds of the covariance function ($j_2(r)=J_0(r)$ and $j_3(r)=\text{sinc}(r)=\frac{\sin(r)}{r}$). In this note I provide the proof for those missing cases, thus completing that of:

\begin{theorem}\label{thm:clt}
    For any order $q\geq 2$ and space dimension $d\geq 2$, the random variable $\frac{h(q,R)}{\sqrt{\var h(q,R)}}$ converges in total variation and 1-Wasserstein distance
    to a standard Gaussian random variable as $R\to\infty$.

    Moreover, if $F(R)=\int_{B_R}f(u_1(x))dx$ is given by a nontrivial local observable with Hermite decomposition
    \begin{equation*}
        f(x)=\sum_{q\geq 2}\frac{a_q}{q!}H_q(x),
        \qquad
        a_q=\frac1{\sqrt{2\pi}}\int_\R f(z)H_q(z) e^{-z^2/2}dz,
    \end{equation*}
    such that either $d=2$ and $a_4\neq 0$ or $\sum_{q\geq 3}\frac{a_q^2C_{q,d}}{(q!)^2}<\infty$,
    then $\frac{F(R)}{\sqrt{\var F(R)}}$ converges in 1-Wasserstein distance to a standard Gaussian random variable as $R\to\infty$.
\end{theorem}

The proofs of \eqref{eq:contractions} in the two missing cases (to be stated again below as \cref{thm:2dcontraction,thm:3dcontraction}) are respectively the content of \cref{sec:2Dcontractions,sec:3Dcontractions}.
For $q=3$, by Fubini's theorem the conditions \eqref{eq:contractions} for $r=1,2$ are both equivalent to
\begin{equation*}
    \norm{K_R}_{L^2(S^{d-1}\times S^{d-1})}^2
    =o(R^{2d}),
\end{equation*}
with
\begin{equation*}
    K_R(a,b)=\frac{1}{\omega_{d-1}^2}\int_{\left(S^{d-1}\right)^2} Q_R(a+\theta+\vartheta) Q_R(b-\theta-\vartheta) d\sigma_{d-1}(\theta) d\sigma_{d-1}( \vartheta).
\end{equation*}
Here the normalization constant $\omega_{d-1}^{-2}$ is clearly irrelevant and I introduce it only for later convenience.
The asymptotic expansion of these integrals is based on a careful control of the asymptotic contribution of different regions of the integration domain.
The proof is slightly different in $d=2$ and 3 because of a dimensional scaling of asymptotics which is best understood inspecting the computation. In the terms introduced by \cite{maininourdin}, the case $d=2$, $q=3$ is of \emph{long memory} type, while $d=3$, $q=3$ is \emph{critical}, and the oscillatory nature of the covariance function plays a relevant role especially in the latter case.

The second statement of \Cref{thm:clt} is deduced by standard arguments \cite{Nourdin2012}. Unlike the case of Berry's model on Euclidean spaces, this result was already fully proved for integral functionals of random spherical harmonics in the high frequency limit \cite{Marinucci2014,Marinucci2015,Rossi2019}.

The following asymptotic relation for constants $C_{q,d}$ provides a useful sufficient condition for the assumption on the decay of coefficients.

\begin{lemma}\label{lem:asymptotic}
    Let $d\geq 2$. As $q\to\infty$, $q\in\N$,
    \begin{equation*}
        \int_0^\infty j_d(r)^q r^{d-1}dr\sim \frac{2^{d/2-1} d^{d/2} \Gamma(d/2)}{q^{d/2}}.
    \end{equation*}
\end{lemma}

In particular the assumption of \cref{thm:clt} is satisfied if 
\begin{equation*}
    |a_q|= O\!\left(\sqrt{q!}\,q^{\frac d4-\frac12-\varepsilon}\right), \qquad q\to\infty,
\end{equation*}
or, as I mentioned above, if $a_4\neq 0$ in $d=2$, because of a change of the overall variance asymptotic behavior.
\cref{lem:asymptotic} can be deduced from a general fact given in \cite[Lemma 1.12]{maini25}, but I shall prefer to prove it using the classical Laplace principle in \cref{sec:clts}.
Under the stronger condition $\sum_{q\geq 3}\frac{a_q^2C_{q,d}}{q!(q-1)!}<\infty$ the normal approximation of \cref{thm:clt} is in the total variation distance, \cref{lem:asymptotic} gives again an analogous sufficient asymptotic condition.

As an example I report an application to a geometric functional that could not be previously treated in $d=2,3$ because its first relevant chaos projection is $q=3$, the \emph{defect} of the scalar field $u_\lambda$.

\begin{theorem}\label{thm:defect}
    As $R\to\infty$, the \emph{defect} or \emph{signed volume} of Berry's model,
    \begin{equation*}
    D(R)=|\set{x\in B_R: u_1(x)>0}|-|\set{x\in B_R: u_1(x)<0}|,
    \quad R>0,    
    \end{equation*}
    satisfies $\var D(R)\sim C R^d$, $C>0$, and $\frac{D(R)}{\sqrt{\var D(R)}}$ converges in 1-Wasserstein distance to a standard Gaussian random variable as $\R\to\infty$.
\end{theorem}

Once again the result can be turned into a $\lambda\to\infty$ limit theorem by scaling, and it follows from \cref{thm:clt} once it is observed that the defect has Wiener chaos decomposition
\begin{equation}\label{eq:defectwiener}
        D(R)=\sum_{q=0}^\infty \frac{\sqrt 2 H_{2q}(0)}{\sqrt\pi (2q+1)!}h(2q+1,R),
\end{equation}
converging in $L^2$ and satisfying the sufficient condition on chaos coefficients displayed above. Since I have omitted chaos projections with $q=1$ above because their Central Limit Theorem is trivial, I recall here for the sake of \cref{thm:defect} that $\var h(1, R)=o\left(R^d\right)$.

The results of this note may be adapted to prove high frequency limit theorems on different underlying  geometrical objects for which the appropriate random wave model is locally asymptotically equivalent to Euclidean ones (as showcased in \cite{dierickx} for nodal statistics).

It is worth mentioning that the Malliavin-Stein approach also allows quantitative estimates on the rate of convergence given the asymptotics of fourth moments, or equivalently of the kernel contractions \eqref{eq:contractions}. In this regard, I simply refer the reader to \cite{Nourdin2012} and opt not to report any such results, because they would just be straightforward applications of the general theory. In a similar fashion, let me remark that the arguments of \cite{maininourdin} allow to obtain Central Limit Theorems for more general isotropic Gaussian random fields (as superpositions of Berry's model at different frequencies) when supplied with the now complete contraction asymptotics.

I will recall the notation as it is used throughout the article, here I only mention that Landau $O$'s and $o$'s have their usual meaning, $\lesssim$ denotes inequality up to a positive constant (independent of any involved parameter) and $\sim$ denotes asymptotic equivalence. 

\begin{acknowledgements}
    I wish to thank G. Peccati for introducing me to the problem considered in this paper, and L. Maini, D. Marinucci and A. P. Todino for many conversations on the topic.
\end{acknowledgements}

\section{2D Contractions}\label{sec:2Dcontractions}

In this Section $d=2$ is fixed, so it is convenient to denote by $\sigma=\frac1{2\pi}\sigma_1$ the uniform distribution on the circle $S^1\subset \R^2$. The covariance of Berry's model is thus $J_0(|z|)=\int_{S^1}e^{i\omega\cdot z}d\sigma(\omega)$, and the contraction estimate to be proved is:

\begin{theorem}\label{thm:2dcontraction}
    Let $d=2$. For $R\to\infty$, $I(R) := \norm{K_R}_{L^2(\sigma^2)}^2\sim \frac{256}{9}R^3$.
\end{theorem}

This Section is devoted to the proof of \cref{thm:2dcontraction}. The first step is a change of variables. 
The squared Bessel function $J_0(|z|)^2$ is the Fourier transform of $\sigma*\sigma$ (the distribution of a 2-steps Pearson random walk, as observed in \cite{Grotto2023b}, $\ast$ denotes convolution):
\begin{equation*}
    J_0(|z|)^2
 =\int_{S^1}\int_{S^1} e^{i(\omega+\omega')z}d\sigma(\omega)d\sigma(\omega')
 =\int_{\R^2}e^{ipz}\rho(p)d p,\quad 
 \rho(p)=\frac{\one_{(0,2)}(|p|)}{\pi^2|p|\sqrt{4-|p|^2}}.
\end{equation*}
On the other hand, the Fourier transform of a ball of radius $R$ is
\begin{equation*}
    Q_R(\xi)=\int_{|x|\leq R}e^{i\xi\cdot x}d x
    =2\pi R |\xi|^{-1} J_1(R|\xi|).
\end{equation*}
By Fubini theorem and the observation above, since $\rho$ and $Q_R$ are radial,
\begin{multline*}
    K_R(a,b)=\int_{\left(S^{d-1}\right)^2} Q_R(a+\theta+\vartheta) Q_R(b-\theta-\vartheta) d\sigma(\theta) d\sigma(\vartheta)\\
    =\int_{\R^2}\rho(p) Q_R(a+p)Q_R(b-p)d p.
\end{multline*}
This expression allows to directly individuate the regions that contribute the most to the integral $K_R$. The function $Q_R(\xi)$ peaks at $\xi=0$, so $K_R(a,b)$ gets most of its value from the regions close to $-a,b$ (where $\rho$ is close to $\rho(1)$) and it is largest when they are antipodal. 

To make this quantitative, recall the standard estimate $|J_1(t)|\lesssim t\wedge t^{-1/2}$ (for all $t>0$), from which
\begin{equation}\label{eq:Hdim2}
    |Q_R(\xi)|=\abs{R^2\frac{2\pi J_1(R|\xi|)}{R|\xi|}}
 \lesssim
 \frac{R^2}{(1+R|\xi|)^{3/2}}, \quad |\xi|>0.
\end{equation}
Observe also that $ Q_R*Q_R=(2\pi)^2Q_R$ (trivially by taking Fourier transforms) so
\begin{equation*}
    \rho(1)\int_{\R^2}Q_R(a+p)Q_R(b-p)dp
 =
 \rho(1)(Q_R*Q_R)(a+b)
 =
 (2\pi)^2\rho(1)Q_R(a+b).
\end{equation*}
I therefore claim the first order expansion:

\begin{lemma}
    As $R\to \infty$, uniformly in $a,b\in S^1$,
    \begin{equation*}
    E_R(a,b):=K_R(a,b)-  \frac{4}{\sqrt 3}Q_R(a+b)=O(R\log R).
    \end{equation*}
\end{lemma}

The numerical constant is simply $(2\pi)^2\rho(1)$.

\begin{proof}
    Split $E_R=E_{in}+E_{out}$ into integrals
\begin{gather*}
    E_{in} := \int_{|p|<2}(\rho(p)-\rho(1))Q_R(a+p)Q_R(b-p)dp,\\
    E_{out}  := -\rho(1)\int_{|p|\geq 2}Q_R(a+p)Q_R(b-p)dp,
\end{gather*}
over the support of $\rho$ and outside of it. I estimate separately the two summands starting from $E_{out}$. If $|p|\ge2$, then $|a+p|,|b-p|\geq |p|/2$ and by \eqref{eq:Hdim2}
\begin{equation*}
    |Q_R(a+p)Q_R(b-p)| \lesssim \tfrac{R}{|p|^3},\quad 
    |E_{out}| \lesssim R\int_{|p|\ge2}|p|^{-3}dp \lesssim R.
\end{equation*}
Consider now the inner region. The relevant contribution should come from neighborhoods of $-a,b$, so let
\begin{equation*}
    N_a:=\{p:|p+a|\le\delta\},
 \quad
 N_b:=\{p:|p-b|\le\delta\},
 \quad 0<\delta<1/4,
\end{equation*}
(possibly overlapping, but only an upper bound is needed) on which $\rho$ and its derivative are bounded and 
\begin{equation*}
     |\rho(p)-\rho(1)|\lesssim |p+a|\quad\text{on }N_a,\qquad
 |\rho(p)-\rho(1)|\lesssim |p-b|\quad\text{on }N_b.
\end{equation*}
Write $\Omega:=\{|p|<2\}\setminus(N_a\cup N_b)$ and decompose further $|E_{in}|\leq |E(N_a)|+|E(N_b)|+|E(\Omega)|$ according to the regions of integration. 

On $N_a$ ($N_b$ is identical), changing variables $x=p+a$ and rescaling,
\begin{align*}
    E(N_a)
    &\lesssim \int_{|x|\le\delta}|x|\,|Q_R(x)|\,|Q_R(a+b-x)|dx\\
    &\lesssim R^4\int_{|x|\le\delta} |x|(1+R|x|)^{-3/2}(1+R|a+b-x|)^{-3/2}dx\\
    &\lesssim R\int_{|y|\le\delta R} 
        \frac{|y|dy}{(1+|y|)^{3/2}(1+|R(a+b)-y|)^{3/2}}\lesssim R\log R.
\end{align*}
The last step follows by splitting once more the integration domain according to the factors in the denominator of the integrand: in fact
\begin{equation*}
        \sup_{z\in\R^2} 
        \int_{|y|\le R} 
        \frac{|y|dy}{(1+|y|)^{3/2}(1+|z-y|)^{3/2}}\lesssim \log R.
\end{equation*}

Consider now $E(\Omega)$, which gives a lower order contribution. Indeed, on $\Omega$ by definition $|a+p|\ge\delta$ and $|b-p|\ge\delta$, so
\begin{equation*}
    |Q_R(a+p)|\lesssim R^{1/2}, \quad
    |Q_R(b-p)|\lesssim R^{1/2}, \quad 
    |Q_R(a+p)Q_R(b-p)|\lesssim R.
\end{equation*}
Since $\int_{|p|\leq 2}|\rho(p)-\rho(1)|dp<\infty$, 
\begin{equation*}
    \int_{\Omega}
 |\rho(p)-\rho(1)|\,|Q_R(a+p)Q_R(b-p)|dp
 \lesssim R,\quad |E_{in}|\lesssim R\log R.\qedhere
\end{equation*}
\end{proof}

The first order term in the expansion of $I(R)$ as $R\to\infty$ should then be given by the integral
\begin{multline*}
    A_R= \int_{S^1}\int_{S^1}Q_R(a+b)^2 d\sigma(a)d\sigma(b)
    =\frac{1}{2\pi}\int_0^{2\pi}  Q_R(e_1+e^{i\theta})^2d\theta\\
    =4 \pi R^3 \int_0^{2 R} \frac{J_1(s)^2}{s^2 \sqrt{1-(s / 2 R)^2}} d s,
\end{multline*}
which is reduced to a one-dimensional integral by rotation invariance, fixing $a=e_1=(1,0)$.
For fixed $s$ the square root in the denominator converges to 1, so splitting the integration domain into $[0, M],[M, R]$, and $[R, 2 R]$ and using once again $|J_1(t)|\lesssim t\wedge t^{-1/2}$ one gets
\begin{equation*}
    \lim_{R\to \infty}\frac{A_R}{4 \pi R^3}=\int_0^{\infty} \frac{J_1(s)^2}{s^2} d s=\frac{4}{3\pi},
\end{equation*}
where the last step is a known integral identity \cite[6.574(2)]{gradrhy}.
The proof of \cref{thm:2dcontraction} is now concluded by controlling the remainder
\begin{equation*}
    I(R)-\frac{16}{3}A_R=O(R^2\log^2 R)+O(R^{5/2}\log R),
\end{equation*}
by the above estimates and Cauchy-Schwarz inequality.

\section{3D Contractions}\label{sec:3Dcontractions}

I will adopt symbols already used in the previous paragraph to indicate the corresponding objects in the 3D case, without any possible ambiguity. For instance, $\sigma=\frac1{4\pi}\sigma_2$ is now the uniform distribution on $S^2\subset \R^3$ and
\begin{equation*}
    \sinc(|z|)=\frac{\sin |z|}{|z|}=\int_{S^2}e^{i\omega\cdot z}d\sigma(\omega)
\end{equation*}
is the covariance function of Berry's model.

\begin{theorem}\label{thm:3dcontraction}
    Let $d=3$. For $R\to\infty$, $I(R) := 
         \norm{K_R}_{L^2(S^2\times S^2)}^2
         \sim 
         2\pi^6R^4$.
\end{theorem}

The remainder of this Section is devoted to the proof of \cref{thm:3dcontraction}.
Since I am using the same notation of above, the objects in the thesis are still
\begin{equation*}
    I(R)=\int_{S^2}\int_{S^2}K_R(a,b)^2d\sigma(a)d\sigma(b),
    \quad 
     K_R(a,b) := \int_{\R^3}\rho(p)Q_R(a+p)Q_R(b-p)d p.
\end{equation*}
In the case $d=3$ the density $\rho$ of $\sigma\ast\sigma$ is particularly simple,
\begin{equation*}
    \sinc(|z|)^2
 =
 \int_{S^2}\int_{S^2}e^{i(\omega+\omega')\cdot z}
 d\sigma(\omega)d\sigma(\omega')
 =
 \int_{\R^3}e^{ip\cdot z}\rho(p)d p,\quad \rho(p)=\frac{\one_{\{|p|<2\}}}{8\pi |p|}
\end{equation*}
(proved in the context of short random walks in \cite{Borwein2012}, it is a direct consequence of known identities of Bessel functions). As for the Fourier transform of the 3D ball,
\begin{equation*}
     Q_R(\xi)= \int_{B_R}e^{i\xi\cdot x}d x
     =R^3F(R\xi),
     \qquad 
     F(s)  =  4\pi\frac{\sin |s|-|s|\cos |s|}{|s|^3},
\end{equation*}
(derived for instance from \cite[2.3 (9)]{Erdelyi1954}).
I shall repeatedly use the inequalities
\begin{equation}\label{eq:usefulineqs}
    |F(s)|\lesssim \frac1{(1+|s|)^2},
    \quad 
     |Q_R(\xi)|\lesssim \frac{R^3}{(1+R|\xi|)^2}.
\end{equation}

The same considerations of Step 2 in \cref{sec:2Dcontractions} lead to the claim
\begin{equation*}
    K_R(a,b)= \pi^2Q_R(a+b)+\text{remainder}.
\end{equation*}
Once can prove a pointwise bound as in \cref{sec:2Dcontractions}, but the remainder now turns out to be $O(R^2)$, as it is revealed by considering the contribution of antipodal $a,b$ (I shall obtain a precise bound below). Mere scaling now makes it impossible to conclude in the same fashion of above, since after taking squares the ``remainder'' would not be negligible compared to the leading term.
The proof in this case must supply local convergence with a further global uniform bound. The need of a fine control of cancellations in oscillatory integrals will of course emerge in the latter, the local part is obtained similarly to the discussion above.

\begin{lemma}[Local Asymptotic]\label{lem:local}
    For all $0<M<\infty$,
    \begin{equation*}
        \sup_{|a+b|\leq M/R} \left|  \frac1{R^3}K_R(a,b)-\pi^2F(R(a+b)) \right|=o(1),\quad R\to \infty.
    \end{equation*}
\end{lemma}

\begin{proof}
    Assume $|R(a+b)|\leq M$, and rewrite
\begin{multline*}
     K_R(a,b)= \int_{\R^3}\rho(s-a)Q_R(s)Q_R(a+b-s)d s\\
    =R^3\int_{\R^3} \rho\left(\frac{t}{R}-a\right) F(t)F(R(a+b)-t)d t,
\end{multline*}
using the definitions and scaling variables.
For fixed $t$ observe that $\rho(t/R-a)\to \rho(-a)=\rho(1)$ uniformly in $a$. 
The thesis is then
\begin{equation}\label{eq:d3localinterm}
    \lim_{R\to \infty}\sup_{|a+b|\leq M/R}
\left|
\int_{\R^3}
\left[
\rho\left(\frac tR-a\right)-\rho(1)
\right]
F(t)F(R(a+b)-t)dt
\right|=0,
\end{equation}
because $F*F=(2\pi)^3F$. The factor $\rho(t/R-a)$ is singular in the moving point $t=Ra$, so I proceed once again by dividing the integration domain into:
\begin{equation*}
    D_1=\{|t|\le A\}, \qquad
    D_2=\{A<|t|\le R/2\},\qquad
    D_3=\{|t|>R/2\},
\end{equation*}
for some $0<A<R/2$, and denoting by $E(D_i)$ the corresponding parts of the integral in \eqref{eq:d3localinterm}.

\emph{Region $D_1$.}
For all $a\in S^2$ and $|t|\leq A$, $1/2\leq |t/R-a|\leq 3/2$. Now $\rho$ is uniformly continuous on this annulus,
\begin{equation*}
    \omega_R(A):=
\sup_{a\in S^2,\,|t|\le A}
\left|
\rho(t/R-a)-\rho(1)
\right|,\quad \lim_{R\to\infty}\omega_R(A)=0.
\end{equation*}
As a consequence,
\begin{equation*}
    \sup_{|a+b|\leq M/R} |E(D_1)|\leq \omega_R(A)
\sup_{|a+b|\leq M/R}
\int_{|t|\le A}|F(t)F(\eta-t)|dt=o(1),\quad R\to\infty.
\end{equation*}
Given this, it is enough to prove that:
\begin{gather*}
    \lim_{A\to\infty} \limsup_{R\to\infty} \sup_{|a+b|\leq M/R} \int_{|t|>A}
\rho\left(\frac tR-a\right) |F(t)F(\eta-t)|dt=0,\\
    \lim_{A\to\infty}\sup_{|\eta|\le M}\int_{|t|>A}|F(t)F(\eta-t)|dt=0.
\end{gather*}
I focus on the first one, the latter is easier and it follows analogously.

\emph{Region $D_2$.} In this region, $\left|\frac tR-a\right|\ge 1-\frac{|t|}{R}\ge \frac12$ so $\rho(t/R-a)\le C$ is bounded, so
\begin{equation*}
    \int_{A<|t|\le R/2}
\rho\left(\frac tR-a\right)
|F(t)F(\eta-t)|dt
\le
C\int_{|t|>A}|F(t)F(\eta-t)|dt .
\end{equation*}
For $A>2M$, uniformly in $|a+b|\le M/B$, on $|t|>A$
\begin{equation*} 
|F(t)F(R(a+b)-t)|
\leq \frac{C_M}{(1+|t|^4)},
 \end{equation*}
and therefore this contribution vanishes as $A\to\infty$.

\emph{Region $D_3$.} This is the region where the moving singularity $t=Ra$ may appear. Using the explicit form of $\rho$,
\begin{multline*}
    \int_{|t|>R/2}
\rho\left(\frac tR-a\right)
|F(t)F(\eta-t)|dt \\
\leq
C_M
\int_{|t|>R/2} 
\frac{\one_{|t/R-a|<2}}{|t/R-a|\, |t|^4}dt
=C_M \frac1R
\int_{|x|>1/2}
\frac{\one_{|x-a|<2} }{|x-a|\, |x|^4}dx .
\end{multline*}
The last integral is uniformly bounded in $a\in S^2$, because $|x|^{-4}$ is bounded on
$\{|x|>1/2\}$ and $|x-a|^{-1}$ is locally integrable in $d=3$.
\end{proof}

\begin{lemma}\label{lem:global}
Uniformly in $a,b\in S^2$ with $a\neq \pm b$,
\begin{equation}\label{eq:globalKR}
    |K_R(a,b)|\lesssim \frac{R^3}{(1+R|a+b|)^{2}}+R\log(2+R).
\end{equation}
\end{lemma}



\begin{proof}
    If $|a+b|\leq \frac1R$, then Cauchy-Schwarz inequality gives
    \begin{equation*}
    |K_R(a,b)|^2
    \leq
    \left(\int \rho(s-a)|Q_R(s)|^2ds\right)
    \left(\int \rho(s-a)|Q_R(a+b-s)|^2ds\right),
    \end{equation*}
    and I contend that both factors are uniformly $O(R^3)$. 
    Consider the first one, the other factor is identical after the translation $s\mapsto a+b-s$, because $|a+b|\leq \frac1R$. Split the integral into $|s|\le R^{-1}$, $R^{-1}<|s|\le 1/2$ and $|s|>1/2$, using
    \begin{equation*}
    |Q_R(s)|\lesssim R^3\one_{\{|s|\le R^{-1}\}}+\frac{R}{|s|^2}\one_{\{|s|>R^{-1}\}},
    \end{equation*}
    and then observe that $|s-a|^{-1}$ is bounded on the first two regions and locally integrable on the last one.

    In the remainder of the proof $|a+b|>\frac1R$ and estimates are understood to be uniform in $a,b$ and $R\geq 2$.
    Subtracting the candidate main contribution in the integral defining $K_R$, let $E_R(a,b)=K_R(a,b)-\pi^2Q_R(a+b)$ with
    \begin{equation*}
        E_R(a,b)=\int_{\R^3}(\rho(s-a)-\rho(1))Q_R(s)Q_R(a+b-s)ds
    \end{equation*}
    just as in \cref{sec:2Dcontractions}, because $Q_R*Q_R=(2\pi)^3Q_R$, the constant coming from $\rho(1)=\frac1{8\pi}$. By \eqref{eq:usefulineqs}, it is sufficient to prove
    \begin{equation}\label{eq:target}
        |E_R(a,b)|\lesssim \frac{R^3}{(1+R|a+b|)^{2}}+R\log(R).
    \end{equation}
    Letting $h=a+b$, a patient but elementary application of trigonometric formulas to the explicit expression of $Q_R$ gives
    \begin{align}\notag
    \frac1{8\pi^2}Q_R(s)Q_R(h-s)
        &=\frac{\cos(R(|s|-|h-s|))-\cos(R(|s|+|h-s|))}{|s|^3|h-s|^3}\\ \notag
        &\quad -R\frac{\sin(R(|s|+|h-s|))+\sin(R(|s|-|h-s|))}{|s|^3|h-s|^2}\\ \notag
        &\quad -R \frac{\sin(R(|s|+|h-s|))-\sin(R(|s|-|h-s|))}{|s|^2|h-s|^3}\\
        \label{eq:relevantQQterm}
        &\quad +R^2 \frac{\cos(R(|s|+|h-s|))+\cos(R(|s|-|h-s|))}{|s|^2|h-s|^2},
    \end{align}
    therefore decomposing $E_R$ according to these summands produces a linear combination of the oscillatory integrals with phases $\Phi_\pm(s):=|s|\mp|h-s|$,
    \begin{equation*}
        Y_{\pm}^{ij}(R)=\int \frac{g_a(s)}{|s|^i |h-s|^j} e^{iR \Phi_\pm(s)} ds,\quad i,j=2,3,\quad g_a(s):=\rho(s-a)-\rho(1),
    \end{equation*}
    each multiplied by a factor proportional to $R^{6-i-j}$.
    The critical sets (portrayed in \cref{fig:placeholder}) are
    \begin{gather*}
        \set{\nabla\Phi_+(s)=0}=\set{s=uh/|h|,\, 0<u<|h|},\\
        \set{\nabla\Phi_-(s)=0}=\set{s=uh/|h|,\, u<0\text{ or }u>|h|}.
    \end{gather*}
    They are both 1-dimensional and the Hessian is non-singular in the normal coordinates: if $s=ue+v$, $e=\frac{h}{|h|}$, $v\in e^\perp$, at $v=0$,
    \begin{equation*}
        \partial^2_{vv}\Phi_+(ue)
        =
        \frac{|h|}{u(|h|-u)}I_{e^\perp},\qquad
        \partial^2_{vv}\Phi_-(ue)
        =
        \begin{cases}
        \frac{|h|}{|u|(|h|-u)}I_{e^\perp}, & u<0,\\
        -\frac{|h|}{u(u-|h|)}I_{e^\perp}, & u>|h|.
        \end{cases}
    \end{equation*}

\begin{figure}
    \centering
    \begin{tikzpicture}[
  x=0.5cm,
  y=0.5cm,
  line cap=round,
  line join=round,
]

\def\rSmall{4.472135955} 
\def\rBig{8.944271910}   

\clip (-9.5,-7) rectangle (17,12);

\draw[black,line width=0.55pt,dotted] (4,2) circle[radius=\rBig];
\draw[black,line width=0.55pt] (0,0) circle[radius=\rSmall];

\draw[black,line width=0.55pt,dashed] (-10,0) -- (0,0);
\draw[black,line width=0.95pt] (0,0) -- (8,0);
\draw[black,line width=0.55pt,dashed] (8,0) -- (18,0);

\fill[black] (0,0) circle[radius=1.25pt];
\node[black,anchor=south west,xshift=1pt,yshift=1pt] at (-1,0) {$0$};

\fill[black] (4,2) circle[radius=1.35pt];
\node[black,anchor=west,xshift=2pt] at (4,2) {$a$};

\fill[black] (4,-2) circle[radius=1.35pt];
\node[black,anchor=west,xshift=2pt] at (4,-2) {$b$};

\fill[black] (8,0) circle[radius=1.35pt];
\node[black,anchor=south east,xshift=-2pt,yshift=1pt] at (9,0) {$a+b$};

\node[black] at (-2.45,4.35) {$S^2$};
\node[black,anchor=west] at (-3.007950,7.211998) {$|s-a|=2$};
\node[black,anchor=west] at (0.7,-0.7) {$\nabla\Phi_+=0$};
\node[black,anchor=west] at (-9,-0.7) {$\nabla\Phi_-=0$};
\node[black,anchor=west] at (13,-0.7) {$\nabla\Phi_-=0$};
\end{tikzpicture}
    \caption{Critical sets of the oscillatory integrals. The solid segment is the critical set of $\Phi_+$, dashed rays are the one of $\Phi_-$.}
    \label{fig:placeholder}
\end{figure}
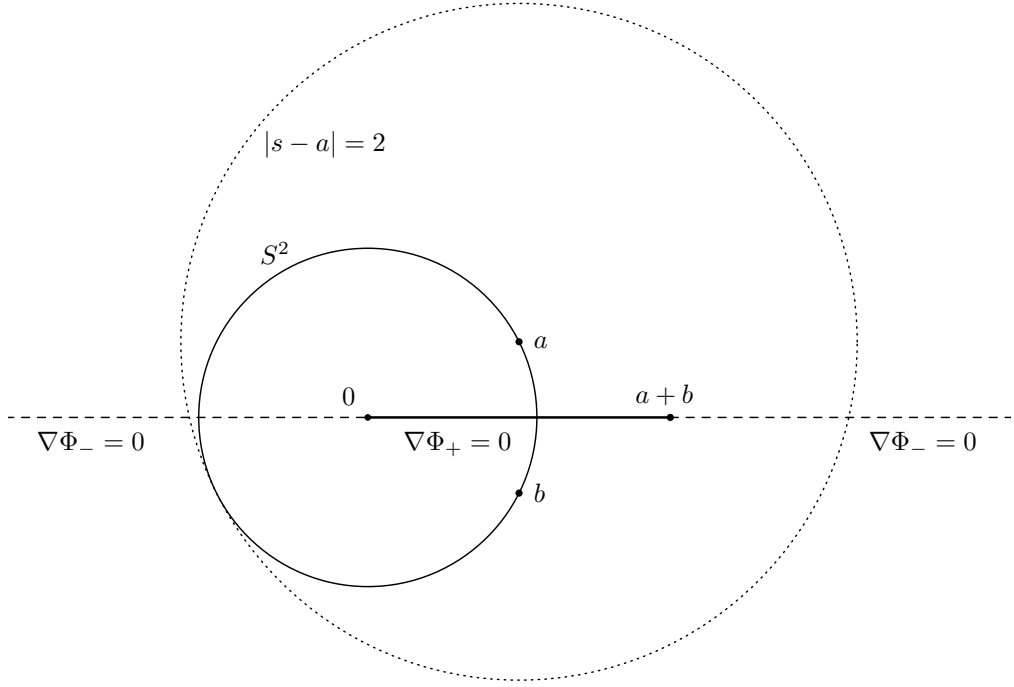

    Before applying the stationary phase principle, some considerations are in order:
    \begin{itemize}
        \item the amplitude factor $g_a(s)$ vanishes at the endpoints of critical regions, therefore the amplitude is overall locally integrable ($|x|^{-2}$ is locally integrable in $d=3$);
        \item when $a,b$ are collinear the amplitude has a singularity on the critical set: there is no need to separately consider these cases because by themselves they are negligible in computing $I(R)$ (hence the additional hypothesis in the statement), however configurations of $a,b$ close to those cases change the overall asymptotic as revealed by the forthcoming computations (and the statement is in fact valid uniformly in \emph{all} $a,b\in S^2$);
        \item amplitudes are not smooth at the extrema of critical sets, but they can be safely excluded since by 
        \eqref{eq:usefulineqs},
        \begin{multline*}
        \int_{|s|\le 2R^{-1}} |g_a(s)|\, |Q_R(s)|\, |Q_R(h-s)|ds\\
        +\int_{|h-s|\le 2R^{-1}} |g_a(s)|\, |Q_R(s)|\, |Q_R(h-s)|ds
        \lesssim \frac{R^3}{(1+R|h|)^{2}};
        \end{multline*}
        \item the same can be said for the region $|s-a|\le R^{-1}$ where the amplitude has a singularity,
        \begin{equation*}
        \int_{|s-a|\le R^{-1}} |g_a(s)|\, |Q_R(s)|\, |Q_R(h-s)|ds
        \lesssim
        R^2\int_0^{1/R} rdr
        \lesssim 1,
        \end{equation*}
        because in this region $|s|,\, |h-s|\gtrsim 1$;
        \item finally, $\rho$ is discontinuous at $|s-a|=2$ but this is irrelevant: a
    $\frac1R$-neighbourhood of the surface $|s-a|=2$ contributes $O(R)$,
    since there $|s|$ and $|h-s|$ are bounded away from zero and
    $|Q_R(s)Q_R(h-s)|\lesssim R^2$.
    \end{itemize}
    Recapitulating, $\frac1R$-neighborhoods of $s=0,a,h$ and $|s-a|=2$ may be excluded from integration in what follows by choosing an appropriate smooth partition of unity, which I shall not formally introduce for notational convenience: derivatives of cutoffs would not appear in the forthcoming estimates.
    The asymptotic contribution of the remaining regions of integration is now entirely determined by the behavior of $\Phi_\pm$ on their critical set.

    I report the asymptotic estimation only for terms with prefactor $R^2$ and amplitude denominator $|s|^2|h-s|^2$. This is because in the remaining integration domain $R|s|\gtrsim 1$ and $R|h-s|\gtrsim 1$, so each of the other factors in the expansion of \eqref{eq:relevantQQterm} is $\lesssim\frac{R^2}{|s|^2|h-s|^2}$ and therefore they produce smaller asymptotic terms. In any case, their asymptotic estimate can be carried out with identical passages.
    The proof is therefore reduced to the asymptotic expansion of
    \begin{equation*}
        Y_{\pm}(R):= Y^{22}_{\pm}(R)=\int e^{iR\Phi_\pm(s)}\frac{g_a(s)}{|s|^2|h-s|^2}ds.
    \end{equation*}
    
    Consider first $\Phi_+$. In the coordinates introduced above, the critical set is parametrized by $0<u<|h|$,
    applying the stationary phase principle \cite[Theorem 7.7.6]{hormander} in the normal variables $v$, with $u$ as parameter, and using the Hessian computed above, gives
    \begin{align}
    \label{eq:criticalplus}
    R^2\left|Y_+(R)\right|
    \lesssim
    R\int_{1/R}^{|h|-1/R}
    \frac{|g_a(ue)|}{u^2(|h|-u)^2}
    \frac{u(|h|-u)}{|h|}du
    + \frac{CR^3}{(1+R|h|)^2},
    \end{align}
    where $u(|h|-u)|h|^{-1}$ is the curvature factor $|\det \partial^2_{vv}\Phi_+(ue)|^{-1/2}$ and the second summand bounds the regions already removed, which rigorously speaking I am still including in the integral $Y_+$, simply for the sake of cleaner notation. 
    For $0<u<|h|$, $|ue-a|^2=1+u(u-|h|)$ and therefore, outside the removed ball $|ue-a|\leq \frac1R$,
    \begin{equation*} 
        |g_a(ue)|=\frac1{8\pi}\left|\frac1{|ue-a|}-1\right|
        \lesssim \frac{u(|h|-u)}{\max\{|ue-a|,R^{-1}\}}.
     \end{equation*}
    Substituting this bound in \eqref{eq:criticalplus} yields
    \begin{align*}
        R^2\left|Y_+(R)\right|
        &\lesssim\frac{R}{|h|}\int_0^{|h|}\frac{du}{\max\{|ue-a|,R^{-1}\}}+\frac{R^3}{(1+R|h|)^2}\\
        &\lesssim R\int_{-2}^{2}\frac{dt}{\max\{|t|,R^{-1}\}}+\frac{R^3}{(1+R|h|)^2}\lesssim R\log(2+R)+\frac{R^3}{(1+R|h|)^2}.
    \end{align*}

    The phase $\Phi_-$ is treated in a similar way. The critical set consists of the two half lines $s=ue$, $u<0$ and $u>|h|$.  Along these half lines, the stationary phase principle (see again \cite[Theorem 7.7.6]{hormander}) in the two normal variables gives
    \begin{align}
    \label{eq:criticalminus}
    R^2\left|Y_-(R)\right|
    \lesssim
    R\int_{(-\infty,0)\cup(|h|,\infty)}
    \frac{|g_a(ue)|}{|u|^2|u-|h||^2}
    \frac{|u|\,|u-|h||}{|h|}du
    + \frac{CR^3}{(1+R|h|)^2}.
    \end{align}
    The referenced result is in fact formulated on a fixed compact integration domain: strictly speaking, since $\set{\nabla\Phi_-(s)=0}$ is not compact, the stationary phase principle must first be applied to compact segments of $u<0$ and $u>|h|$, excluding also the intersection with $|s-a|=2$. The constants are uniform on these pieces once the singular neighborhoods already estimated above have been removed, and summing over the pieces gives precisely the one-dimensional integral \eqref{eq:criticalminus}.
    I already observed that $|ue-a|^2-1=u(u-|h|)$, using now $|g_a(ue)|\lesssim \min(|u|\,|u-|h||,1)$ and changing variables ($u=-t$ and $u=|h|+t$ on the two half lines),
    \begin{equation*}
        \int_{(-\infty,0)\cup(|h|,\infty)} \frac{|g_a(ue)|}{|u|\,|u-|h||}du
        \lesssim \int_0^\infty \frac{\min\{t(t+|h|),1\}}{t(t+|h|)}dt \lesssim 1.
    \end{equation*}
    Hence \eqref{eq:criticalminus} gives
    \begin{equation*} 
    R^2\left|Y_-(R)\right|
    \lesssim
    \frac{R}{|h|}+\frac{R^3}{(1+R|h|)^2},
     \end{equation*}
    which is bounded by the right-hand side of \eqref{eq:target}.
\end{proof}

I now combine the two estimates. Observe first that $K_R(a,b)$ is invariant under rotations, so it only depends on $a,b$ through $a\cdot b$ or equivalently $|a+b|$, so I write $K_R(a,b)=K_R(|a+b|)$. Then, changing variables,
\begin{multline*}
    \int_{S^2}\int_{S^2}
 \left|K_R(a,b)-\pi^2Q_R(a+b)\right|^2
 d\sigma(a)d\sigma(b)\\
 =\frac12 \int_0^2\left|K_R(r)-\pi^2 Q_R(r)\right|^2 r dr
 =\frac{R^4}{2} \int_0^{2 R}\left|\tfrac{1}{R^3} K_R\left(\tfrac{s}{R}\right)-\pi^2 F(s)\right|^2 s d s
\end{multline*}
is now a 1-dimensional integral.
Let $G_R(s):=\frac{1}{R^3}K_R\left(\frac{s}{R}\right)-\pi^2F(s)$ be the integrand in the last expression, which I claim is $o(R^4)$.
By \cref{lem:local}, for every $M>0$
\begin{equation}\label{eq:local-part-GR}
    \int_0^M |G_R(s)|^2sds
    \le \frac{M^2}{2}\sup_{0\le s\le M}|G_R(s)|^2=o(1).
\end{equation}
For the tail, \cref{lem:global} gives, uniformly in $0\le s\le 2R$,
\begin{equation*}
    \left|\frac1{R^3}K_R\left(\frac{s}{R}\right)\right|
    \lesssim \frac1{(1+s)^2}+\frac{\log(2+R)}{R^2}.
\end{equation*}
Since $|F(s)|\lesssim \frac1{(1+s)^2}$, it follows $|G_R(s)|\lesssim \frac1{(1+s)^2}+\frac{\log(2+R)}{R^2}.$
As a consequence,
\begin{align*}
    \int_M^{2R}|G_R(s)|^2sds
    &\lesssim
    \int_M^{\infty}\frac{s}{(1+s)^4}ds
    +\frac{\log(2+R)}{R^2}\int_M^{2R}\frac{s}{(1+s)^2}ds
    +\frac{\log(2+R)^2}{R^4}\int_0^{2R}sds \notag\\
    &\lesssim \int_M^{\infty}\frac{s}{(1+s)^4}ds
    +\frac{\log(2+R)^2}{R^2}.
\end{align*}
Taking first $\limsup_{R\to\infty}$ and then letting $M\to\infty$ gives the desired approximation in $L^2$. I now use it to derive the asymptotic of $I(R)$.
Expand:
\begin{multline}\label{eq:IR-expansion}
    I(R)=\norm{K_R}_{L^2(S^2\times S^2)}^2
    =\pi^4 \int_{S^2}\int_{S^2}Q_R(a+b)^2d\sigma(a)d\sigma(b)\\
    +2\pi^2 \int_{S^2}\int_{S^2}Q_R(a+b)E_R(a,b)d\sigma(a)d\sigma(b)
    +\norm{E_R}_{L^2(S^2\times S^2)}^2.
\end{multline}
The last term is \(o(R^4)\) by the \(L^2\)-approximation just proved,
which uses \cref{lem:local} on compact \(s\)-intervals and
\cref{lem:global} to control the tail. Moreover, by Cauchy-Schwarz inequality,
\begin{multline*}
    \left|\int_{S^2}\int_{S^2}Q_R(a+b)E_R(a,b)d\sigma(a)d\sigma(b)\right|\\
    \leq
    \left(\int_{S^2}\int_{S^2}Q_R(a+b)^2d\sigma(a)d\sigma(b)\right)^{1/2}
    \|E_R\|_{L^2(S^2\times S^2)}.
\end{multline*}
The second factor is $o(R^2)$ so if the first one is $O(R^2)$ the cross term is $o(R^4)$, but
\begin{equation*}
    \int_{S^2}\int_{S^2}Q_R(a+b)^2d\sigma(a)d\sigma(b)
    =\frac12\int_0^2 Q_R(r)^2rdr
    =\frac{R^4}{2}\int_0^{2R}sF(s)^2ds\sim 2\pi^2R^4,
\end{equation*}
where the last step follows from the known identity \cite[6.574(2)]{gradrhy} entailing
\begin{equation*}
    \int_0^\infty sF(s)^2ds
    =16\pi^2\int_0^\infty \frac{j_1(s)^2}{s}ds
    =4\pi^2.
\end{equation*}

\section{Bessel Moments}\label{sec:clts}
This Section entirely consists of the proof of \cref{lem:asymptotic}.
Recall that the Bessel function of the first kind $J_\nu(z)$ is the real analytic function
\begin{equation}\label{eq:besselDEF}
	J_\nu(z)=\left(\frac{z}{2}\right)^\nu \sum_{k=0}^{\infty} \frac{\left(-z^2/4\right)^k}{k ! \Gamma(\nu+k+1)},\quad \nu\in[0,\infty)
\end{equation}
(see \cite[10.2.2]{olver}). As a consequence,
\begin{equation*}
j_d(r)=1-\frac{r^2}{2 d}+O\left(r^4\right),\quad
q \log j_d(r / \sqrt{q})=-\frac{r^2}{2 d}+O\left(r^4/q\right),
\end{equation*}
uniformly for $0\leq r\leq M$ for fixed $M>0$.
Take $\epsilon>0$ smaller than the first zero of $j_d$, by Laplace's principle, as $q\to\infty$,
\begin{multline*}
    \int_0^{\epsilon} j_d(r)^q r^{d-1} d r=q^{-d / 2} \int_0^{\epsilon{\sqrt{q}}} j_d(s / \sqrt{q})^q s^{d-1} d s\\ 
    \sim 
    q^{-d/2}\int_0^{\infty} e^{-s^2 /(2 d)} s^{d-1} d s
    =q^{-d/2}2^{d / 2-1} d^{d / 2} \Gamma(d / 2).
\end{multline*}
I contend that this is the relevant first-order contribution, that is the complement is negligible. 
Indeed, $|j_d|\leq1$ uniformly (it is the Fourier transform of a probability measure) and the usual decay of Bessel functions gives $|j_d(r)| \leq C_d(1+r)^{-(d-1) / 2}$, therefore
\begin{equation*}
\int_{\epsilon}^{\infty}\left|j_d(r)\right|^{q} r^{d-1} d r=o\left(q^{-d / 2}\right),\quad q\to \infty,
\end{equation*}
by splitting integration over $[\epsilon,M]$ and $(M,\infty)$ for fixed and large enough $M$.


\end{document}